\documentclass{amsart}
\usepackage{amsfonts,amssymb,amsmath,amsthm}
\usepackage{enumerate}

\usepackage{graphicx, color}

\makeatletter
\@namedef{subjclassname@2020}{\textup{2020} Mathematics Subject Classification}
\makeatother

\makeatletter

\newtheorem*{thrm}{Theorem}
\newtheorem{lem}{Lemma}
\newtheorem{prop}{Proposition}
\newtheorem*{cor}{Corollary}
\newtheorem{definition}{Definition}

\newtheorem{remark}{Remark}
\newtheorem{ex}{Example}
\numberwithin{equation}{section}

\begin{document}
\title[L-shadowing lemma for the Cauchy equation]{L-shadowing lemma for the Cauchy equation}

\author[K. Lee]{K. Lee}
\address{Department of Mathematics, Chungnam National University, Daejeon, 34134, Republic of Korea.}
\email{khlee@cnu.ac.kr}

\author[C.A. Morales]{C.A. Morales}
\address{Beijing Advanced Innovation Center for Future Blockchain and Privacy Computing, Beihang University, Beijing, China.}
\email{morales@impa.br}



\begin{abstract}
We prove that if the Cauchy problem $\dot{u}=Au$ in a Banach space is hyperbolic, then the problem has the L-shadowing property. Conversely, if the space is finite-dimensional and the L-shadowing property is satisfied, then the problem is hyperbolic. This generalizes a previous result by Ombach \cite{o, o1} for linear homeomorphisms. Some short applications are given.
\end{abstract}

\maketitle




\section{Introduction}

\noindent
The concept of the {\em shadowing property} was originally introduced by Bowen in the discrete case \cite{b}, and its extension to flows was studied by Franke and Selgrade \cite{fs}, as well as later by Thomas \cite{t}. This property holds a fundamental role within the dynamical literature, as demonstrated, for instance, in \cite{kh}. A variation known as the {\em L-shadowing property}, initially termed the "S-limit shadowing property" by Lee and Sakai \cite{ls}, was subsequently independently rediscovered and named by Carvalho and Cordeiro \cite{cc}. Research into the shadowing property of partial differential equations (PDEs) has been undertaken. For instance, a study by Pilyugin \cite{pil} demonstrated this property in relation to the time-$t$ mapping of the Chafee-Infante equation $u_t=\Delta u+bu+f(u)$ on the open interval $(0,\pi)$. This was observed under Dirichlet conditions in $H^1_0(0,\pi)$ within a vicinity of the global attractor. However, as of now, no established outcomes exist regarding the L-shadowing property within the realm of PDEs.

In this paper, we initiate the study of the L-shadowing property in PDEs through the abstract Cauchy problem
\begin{equation}
\label{eq1}
\dot{u}=Au
\end{equation}
in a (complex) Banach space $X$. Specifically, we address the conditions under which it exhibits the L-shadowing property. Our main results establish that if the problem \eqref{eq1} is hyperbolic, it inherently satisfies the L-shadowing property. Furthermore, in cases where the L-shadowing property is met and the dimension of $X$ is finite, the problem is hyperbolic. These findings draw inspiration from analogous results obtained by Ombach in the realm of linear homeomorphisms \cite{o,o1}. To provide practical context, we also offer a selection of applications stemming from these results. Hereafter, we proceed to elaborate on these contributions in detail. 
 
A \emph{semiflow} of a metric space $X$ is a one-parameter family of continuous maps $T =\{T(t) : X \to X\}_{t \geq 0}$ that satisfies the following conditions: $T(0) = I$ (the identity map on $X$), and $T(s + t) = T(s) \circ T(t)$ for all $s, t \geq 0$. The individual map $T(t)$ for a given $t \geq 0$ is called the \emph{time $t$-map}.

The shadowing property for semiflows, which we will adopt here, is a natural forward-in-time modification of the corresponding definition for flows introduced by Komuro \cite{k}. This definition is similar to the ones used by Grag and Das \cite{gd}, as well as by \cite{hwl}. For the original definition of shadowing for flows, refer to Franke, Selgrade and Thomas \cite{fs,t}.

Let $T$ be a semiflow on a metric space $X$.
Given $\delta,R>0$ a sequence $(x_i,t_i)_{i\geq0}\subset X\times (0,\infty)$ is called $(\delta,R)$-pseudo orbit of $T$ if
$$
t_i\geq R\quad\mbox{ and }\quad
d(T(t_i)x_i,x_{i+1})<\delta,\quad\forall i\geq0.
$$
Define
$$
\hat{t}_0=0
\quad\mbox{ and }\quad
\hat{t}_i=\sum_{j=0}^{i-1}t_j,\quad\quad\mbox{ for }i>0
$$
We introduce the useful notation
$$
x_0 * t=\phi_{t-\hat{t}_i}(x_i),\quad\quad\forall  i\geq0,\,\hat{t}_i\leq t<\hat{t}_{i+1}.
$$

\begin{definition}
A semiflow $T$ has the {\em shadowing property} if for every $\epsilon>0$ there are $\delta, R>0$ such that every $(\delta, R)$-pseudo orbit of $T$ there are $x\in X$
and an increasing homeomorphism $h:[0,\infty)\to [0,\infty)$ such that
$$
d(T(h(t))x,x_0 * t)<\epsilon\quad\quad\forall t\geq0.
$$
\end{definition}

The following definition is the forward-in-time variation of the L-shadowing property proposed in \cite{dln,ls}.

\medskip

\begin{definition}
A semiflow $T$ has the {\em L-shadowing property} if for every $\epsilon>0$ there are $\delta, R>0$ such that every
$(\delta, R)$-pseudo orbit $(x_i,t_i)_{i\geq0}$ of $T$
satisfying
$$
\lim_{i\to\infty}d(T(t_i)x_i,x_{i+1})=0,
$$
there are $x\in X$ and an increasing homeomorphism $h:[0,\infty)\to[0,\infty)$ such that
$$
d(T(h(t))x,x_0 * t)<\epsilon\quad\quad(\forall t\geq 0)
\quad\mbox{ and }\quad
\lim_{t\to\infty}d(T(t)x,x_0 * t)=0.
$$
\end{definition}

Now, we focus on the case when $X$ is a Banach space with the norm $\|\cdot\|$.
By a {\em semigroup of linear operators} of $X$ we mean a semiflow
$T$ of $X$ such that $T(t)$ is a bounded linear operator of $X$, for all $t\geq0$.
We can associate a map $A:D(A)\to X$ with domain $D(A)$ defined by
$$
D(A)=\{x\in X:\lim_{t\to0}\frac{T(t)x-x}t\mbox{ exists}\}\mbox{ and }
A(x)=\lim_{t\to0}\frac{T(t)x-x}t,\quad\forall x\in D(A).
$$
It is clearly linear but not necessarily bounded. It is often called
{\em infinitesimal generator} of $T$.
We say that $T$ is a {\em $C_0$-semigroup} if
$$
\lim_{t\to0}\|T(t)x-x\|=0,\quad\quad\forall x\in X.
$$

We will assume that the linear operator $A$ in the Cauchy problem \eqref{eq1} is the infinitesimal generator of a $C_0$-semigroup of linear operators of $X$.
The standard conditions for this are (Theorem 5.3 in \cite{pazy}):
$A$ is closed, $D(A)$ is dense in $X$, and the resolvent set of $A$ contains the ray $(\omega,\infty)$ for some $\omega>0$. Additionally, there exists $M>0$ such that
$$
\|(\lambda I-A)^{-n}\|\leq \frac{M}{(\lambda-\omega)^n},\quad\quad
\forall \lambda\in (\omega,\infty),\, n\in\mathbb{N}.
$$
It should be noted that the L-shadowing property may hold or not for a given $C_0$-semigroup. To illustrate this, consider the following two examples.

\medskip

\begin{ex}
The \emph{trivial semigroup} (i.e., $T(t) = I$ for every $t \geq 0$) does not have the L-shadowing property. This can be concluded from the arguments presented in the proof of Theorem 2.3.2 in \cite{ah}.
On the other hand, the semigroup $e^{tA}$ generated by a bounded linear operator $A: X \to X$, where the spectrum of $A$ belongs to $\{z \in \mathbb{C} : \text{Im}(z) < 0\}$, does have the L-shadowing property. This can be deduced from Lemma \ref{federal} below and Theorem 4.3 p. 118 in \cite{pazy} (compare with the shadowing property case in Remark 3.4 of \cite{gd} or \cite{r}).
\end{ex}

\medskip

Based on this example we introduce the following auxiliary concept.

\medskip

\begin{definition}
The Cauchy problem \eqref{eq1} is said to have the L-shadowing property
if $A$ is the infinitesimal generator of a $C_0$-semigroup with the L-shadowing property.
\end{definition}

\medskip

Similarly we define the shadowing property for the Cauchy problem.
This definition raises the question of whether spectral properties on $A$ can be used to characterize when the corresponding Cauchy problem has the L-shadowing properties. We provide a partial answer using the following concepts.

Recall that an operator $B:X\to X$ on a Banach space $X$ is {\em hyperbolic} if its spectrum $\sigma(B)$ does not intersect the unit complex circle $S^1$. Drawing upon references such as \cite{bvr}, \cite{bt}, \cite{en}, or Definition 1 in \cite{ra}, we adopt the following definition.

\medskip

\begin{definition}
A semigroup of linear operators $T$ on a Banach space is {\em hyperbolic} if its time-one map $T(1)$ is a hyperbolic operator. We define the Cauchy problem \eqref{eq1} to be hyperbolic if $A$ generates a hyperbolic $C_0$-semigroup.
\end{definition}

\medskip

(Note the difference between this notion of hyperbolicity and that for evolution equations in Section 5.3 in \cite{pazy}.)

We are now in position to state our result. It gives the desired necessary and sufficient conditions for the Cauchy problem \eqref{eq1} to have both the shadowing and L-shadowing properties.

\medskip

\begin{thrm}
\label{thA}
If the Cauchy problem \eqref{eq1} in a Banach space $X$ is hyperbolic, then the problem has the L-shadowing property.
Conversely, if the problem has the L-shadowing property and $dim(X)<\infty$, then the problem is hyperbolic.
\end{thrm}

\medskip

A suitable remark is as follows.

\medskip

\begin{remark}
\label{walker}
It is worth to mention that the analogous result obtained by replacing L-shadowing by shadowing in the above statement holds too
(see \cite{g} for a possible reference about this fact).
\end{remark}

\medskip

The proof of the theorem also implies that all hyperbolic $C_0$-semigroups have the strong shadowing property (as defined in Komuro \cite{k}).
An interesting problem that arises from our investigation is whether there exists a relationship between the shadowing and L-shadowing properties in the context of the Cauchy problem. Specifically, we anticipate that every Cauchy problem with the shadowing property also possesses the L-shadowing property. This expectation is motivated by the observation that this implication holds true for linear homeomorphisms (see \cite{lm}).

This paper is divided as follows.
In Section \ref{sec2} we prove three lemmas.
In Section \ref{sec3} we prove the theorem and in Section \ref{sec4} we give some applications.

\section{Preliminary lemmas}
\label{sec2}

\noindent
In this section we prove three lemmas.
To motivate the first, we say that
a semiflow $T$ of a metric space $X$ is
{\em uniformly contracting} \cite{gd} if
there is $\lambda>0$ such that
$d(T(t)x,T(t)y)\leq e^{-t\lambda}d(x,y)$,
for all $x,y\in X,\, t\geq0.$
It was mentioned in Remark 3.4 of \cite{gd} that every uniformly contracting semiflow of a metric space has the shadowing property.
The following lemma deals with a slight more general class of semiflows:

\medskip

\begin{lem}
\label{federal}
Let $T$ be a semiflow of a complete metric space $X$. Suppose that there are
$K,\lambda>0$ such that
$$
d(T(t)x,T(t)y)\leq Ke^{-\lambda t}d(x,y),\quad\quad\forall x,y\in X, t\geq0.
$$
Then, for every $\epsilon>0$ there are $\delta,R>0$ such that
for every $(\delta,R)$-pseudo orbit $(x_i,t_i)_{i\geq0}$ satisfying
$$
\lim_{i\to\infty}d(T(t_i)x_i,x_{i+1})=0
$$
there is $x\in X$ such that
$$
d(T(t)x,x_0 * t)\leq\epsilon
\quad
(\forall t\geq0)\quad\mbox{ and }\quad
\lim_{t\to\infty}d(T(t)x,x_0 * t)=0.
$$
\end{lem}

\begin{proof}
We will adapt the proof of Proposition 1 in Ombach \cite{o} to the semiflow case.
Let $K,\lambda>0$ be given by the fact that $T$ is eventually contractive.
Let $\epsilon>0$. Take $R>0$ such that
$Ke^{-\lambda R}<1$. Define
$$
\delta=(1-Ke^{-\lambda R})\frac{\epsilon}K.
$$
Let $(x_i,t_i)_{i\geq0}$ be a $(\delta,R)$-pseudo orbit namely
$$
t_i\geq R\quad\mbox{ and }\quad
d(T(t_i)x_i,x_{i+1})<\delta,\quad\forall i\geq0.
$$
Define
$$
E=\{y=(y_i)_{i\geq0}:y_i\in X\mbox{ and }d(x_i,y_i)\leq \frac{\epsilon}K,\, \forall i\geq0\}.
$$
Since $X$ is complete, $E$ also is if endowed with the supremum metric
$$
D(y,z)=\sup_{i\geq0}d(y_i,z_i).
$$
For any $y\in E$ we define the sequence
$\Gamma(y)=(\Gamma(y)_i)_{i\geq0}$ by

$$
\Gamma(y)_i = \left\{
\begin{array}{rcl}
x_0, & \mbox{if} & i=0\\
T(t_{i-1})y_{i-1}, & \mbox{if} & i>0.
\end{array}
\right.
$$
For any $y\in E$ one has
$d(x_0,\Gamma(y)_0)=d(x_0,x_0)=0\leq \frac{\epsilon}K$ and
for all $i\geq0$,
\begin{equation}
\label{common}
d(x_i,\Gamma(y)_i) = d(x_i,T(t_{i-1})y_{i-1})
\leq d(x_i,T(t_{i-1})x_{i-1})+
d(T(t_{i-1})x_{i-1},T(t_{i-1})y_{i-1}).
\end{equation}
This implies
$$
d(x_i,\Gamma(y)_i) 
\leq \delta+Ke^{-\lambda t_{i-1}}d(y_{i-1},x_{i-1})\\
\leq (1-Ke^{-\lambda R})\frac{\epsilon}K+Ke^{-\lambda R}\frac{\epsilon}K\\
= \frac{\epsilon}K, \quad\quad\forall i>0,
$$
thus $\Gamma(y)\in E$ for all $y\in E$.
So, we obtain a map
$\Gamma:E\to E$.

Since $Ke^{-\lambda R}<1$ and
\begin{eqnarray*}
D(\Gamma(y),\Gamma(z)) & = &
\sup_{i\geq0}d(\Gamma(y)_i,\Gamma(z)_i)\\
& = &
\sup_{i\geq1}d(T(t_{i-1})y_{i-1},T(t_{i-1})z_{i-1})\\
& \leq & Ke^{-\lambda t_{i-1}} \sup_{i\geq1}d(y_{i-1},z_{i-1})\\
& \leq & Ke^{-\lambda R}D(y,z)
\end{eqnarray*}
for all $y,z\in E$, we have that $\Gamma:E\to E$ is a contraction. Since $E$ is a complete metric space,
the Banach contracting principle implies that
$\Gamma$ has a fixed point i.e. there is
$y\in \Gamma$ such that
$\Gamma(y)=y$.
It follows that
\begin{equation}
\label{dady}
y_0=x_0,\quad
d(y_i,x_i)\leq \frac{\epsilon}K\quad\mbox{ and }\quad
T(t_{i-1})y_{i-1}=y_i,\quad\forall i\geq1.
\end{equation}
The first and third of the above expressions imply
$$
y_i=T(\hat{t}_i)x_0,\quad\quad\forall i\geq0.
$$
By replacing in the second we obtain
$$
d(T(\hat{t}_i)x_0,x_i)\leq\frac{\epsilon}K,\quad\quad\forall i\geq1.
$$
Then,
\begin{equation}
\label{esquisito}
d(T(t)x_0,T(t-\hat{t}_i)x_i)\leq Ke^{-\lambda (t-t_i)}d(T(\hat{t}_i)x_0,x_i)\leq K\frac{\epsilon}K=\epsilon
\end{equation}
for all $i\geq0$ and $\hat{t}_i\leq t<\hat{t}_{i+1}$
so
$$
d(T(t)x_0,x_0*t)\leq \epsilon,\quad\quad\forall t\geq0.
$$
Define
$$
E_0=\{y\in E:\lim_{i\to\infty}d(y_i,x_i)=0\}.
$$
It follows that $E_0$ is a closed subspace of $E$.
Applying \eqref{common} we obtain
$$
\lim_{i\to\infty}d(x_i,\Gamma(y)_i)
\leq \lim_{i\to\infty}[d(x_i,T(t_{i-1})x_{i-1})+d(x_{i-1},y_{i-1})]=0,
\quad\forall y\in E_0,\, i\geq 0.
$$
It follows that $\Gamma(E_0)\subset E_0$.
Since $E_0$ is closed, the fixed point $y$ in \eqref{dady} belongs to $E_0$.
Then, \eqref{esquisito}
implies
$$
\lim_{t\to\infty}d(T(t)x_0,x_0*t)=0
$$
completing the proof.
\end{proof}

The second lemma is inspired by Proposition 3 in Ombach \cite{o}.

\medskip

\begin{lem}
\label{abe}
Let $T$ be a semigroup of a Banach space $X$
such that $T(t)$ is invertible for all $t\geq0$.
Suppose that there are $\lambda,K>0$ such that
$$
\|(T(t))^{-1}\|\leq Ke^{-t \lambda},\quad\quad\forall t\geq0.
$$
Then, for every $\epsilon>0$ there is $\delta>0$
such that for every $(\delta,1)$-pseudo orbit $(x_i,t_i)_{i\geq0}$
satisfying
$$
\lim_{i\to\infty}\|T(t_i)x_i-x_{i+1}\|=0,
$$
there is $x\in X$ such that
$$
\|T(t)x-x_0*t\|\leq\epsilon\quad(\forall t\geq0)\quad\mbox{ and }\quad
\lim_{t\to\infty}\|T(t)x-x_0 * t\|=0.
$$
\end{lem}

\begin{proof}
Fix $\epsilon>0$.
Take $\delta>0$ such that
\begin{equation}
\label{coker}
K\cdot\left(\sum_{i=0}^\infty(e^{-\lambda})^k\right)\cdot\delta<\epsilon.
\end{equation}

Let $(x_i,t_i)_{i\geq0}$ be a $(\delta,1)$-pseudo orbit namely
$$
t_i\geq 1\quad\mbox{ and }\quad
\|T(t_i)x_i-x_{i+1}\|\leq\delta,\quad\forall i\geq0.
$$
It follows that
\begin{equation}
\label{negan}
x_{i+1}=T(t_i)x_i+h_i,\quad\quad\forall i\geq0,
\end{equation}
for some sequence $(h_i)_{i\geq0}$ satisfying
$\|h_i\|\leq\delta$ for all $i\geq0$.
By induction,
$$
x_i=\left(\prod_{j=0}^{i-1}T(t_j)\right)x_0+\sum_{k=1}^i\left(
\prod_{j=k}^{i-1}T(t_j)\right)h_{k-1},
\quad\quad\forall i\geq0
$$
(here we convey that $\sum_a^b=I$ when $a>b$).
Then,
$$
x_i=T(\hat{t}_i)[x_0+\sum_{k=1}^i(T(\hat{t}_k)^{-1})h_{k-1}],
\quad\quad\forall i\geq0.
$$
Since
$$
\|(T(\hat{t}_k)^{-1})h_{k-1}\|\leq \|(T(\hat{t}_k))^{-1}\|\|h_{k-1}\|
\leq Ke^{-\lambda \hat{t}_k}\delta
= Ke^{-\lambda\sum_{j=0}^{k-1}t_j}\delta
 \leq  K(e^{-\lambda})^k\delta
$$
and $e^{-\lambda}<1$, the series
$$
x=x_0+\sum_{k=1}^\infty(T(\hat{t}_k)^{-1})h_{k-1}
$$
is convergent.
Now, fix $i\geq0$ and $\hat{t}_i\leq t<\hat{t}_{i+1}$.
Then,
\begin{eqnarray*}
T(t)x-T(t-\hat{t}_i)x_i&=&T(t)x_0+T(t)\sum_{k=1}^\infty
(T(\hat{t}_k)^{-1})h_{k-1}-T(t)
\left[x_0+\sum_{k=1}^i(T(\hat{t}_k)^{-1})h_{k-1}\right]\\
&=& T(t)
\sum_{k=i+1}^\infty
(T(\hat{t}_k)^{-1})
h_{k-1}\\
& = & \sum_{k=i+1}^\infty (T(\hat{t}_k-t))^{-1}h_{k-1}
\end{eqnarray*}
so
\begin{equation}
\label{atorney}
\|T(t)x-T(t-\hat{t}_i)x_i\|
\leq \left(\sum_{k=i+1}^\infty \|(T(\hat{t}_k-t))^{-1}\|\right)\cdot \sup_{k\geq i}\|h_{k}\|
\end{equation}
Since
\begin{eqnarray*}
\sum_{k=i+1}^\infty\|(T(\hat{t}_k-t)^{-1}\| &\leq& \sum_{k=i+1}^\infty Ke^{-\lambda (\hat{t}_k-t)}\\
&<& K\sum_{k=i+1}^\infty e^{-\lambda(\hat{t}_k-\hat{t}_{i+1})}\\
& \leq & K\sum_{k=i+1}^\infty e^{-\lambda(k-i-1)}\\
& = & K\sum_{k=0}^\infty (e^{-\lambda})^k,
\end{eqnarray*}
\eqref{atorney} implies
$$
\|T(t)x-T(t-\hat{t}_i)x_i\|
\leq K\cdot \left(\sum_{k=0}^\infty (e^{-\lambda})^k\right)\cdot \delta\overset{\eqref{coker}}{<}\epsilon.
$$
This proves
$$
\|T(t)x-x_0*t\|<\epsilon,\quad\quad\forall t\geq0.
$$
Since
$$
\lim_{i\to\infty}\|T(t_i)x_i-x_{i+1}\|=0,
$$
the sequence $(h_i)_{i\geq0}$ in
\eqref{negan} satisfies
$$
\lim_{i\to\infty}h_i=0.
$$
Fix $\Delta>0$ and $i_\Delta>0$ large so that
$$
K\cdot\left(\sum_{k=0}^\infty (e^{-\lambda)})^k\right)\cdot \sup_{k\geq i}\|h_k\|<\Delta,\quad\quad\forall i\geq i_\Delta.
$$
Then, \eqref{atorney} implies $
\|T(t)x-T(t-\hat{t}_i)x_i\| <\Delta$ for all $i\geq i_\Delta$ and $\hat{t}_i\leq t<\hat{t}_{i+1}$
hence
$$
\lim_{t\to\infty}\|T(t)x-x_0*t\|=0
$$
completing the proof.
\end{proof}

\medskip

The third lemma is a direct consequence of the definition. Its proof is included for completeness.

\medskip

\begin{lem}
\label{lanza}
Let $T$ be a semigroup of linear operators of a Banach space $X$.
If $T$ has the L-shadowing property, and there is a direct sum $X=M\oplus N$ by closed subspaces $M, N$ such that
$T(t)M\subset M$ and $T(t)N\subset N$ for all $t\geq0$, then the restricted semigroups
$T|_M$ and $T|_N$ have the L-shadowing property.
\end{lem}

\begin{proof}
Denote by $T^M=T|_M$ and $T^N|_N$ the restrictions of $T$ to $M$ and $N$ respectively.
We only need to prove that $T^M$ has the shadowing property.
By passing to projections if necessary we can assume
\begin{equation}
\label{anis}
\|x\|=\max\{\|x^M\|,\|x^N\|\}\quad\mbox{ whenever }\quad x=x^M\oplus x^N\in M\oplus N.
\end{equation}
Let $\epsilon>0$ and $\delta,R>0$ be given by the L-shadowing property of $T$.
Let $(x^M_i,t_i)_{i\geq0}$ be a $(\delta,R)$-pseudo orbit of $T^M$ satisfying
$\lim_{i\to\infty}\|T^M(t_i)x_i-x_{i+1}\|=0$.
It follows from \eqref{anis} that
it is also a $(\delta,R)$-pseudo orbit of $T$ satisfying
$\lim_{i\to\infty}\|T(t_i)x_i-x_{i+1}\|=0$.
Then, there are $x\in X$ and an increasing homeomorphism $h:[0,\infty)\to [0,\infty)$ such that
\begin{equation}
\label{daryl}
\|T(h(t))x-x^M_0*t\|\leq \epsilon\quad\quad(\forall t\geq0)
\quad\mbox{ and }\quad
\lim_{t\to\infty}\|T(h(t))x-x^M_0*t\|=0.
\end{equation}
Writing $x=x^M+x^N$ one has
\begin{eqnarray*}
\|T^M(h(t))x^M-T^M(t-\hat{t}_i)x^M_i\|&\leq& \max\{
\|T^M(h(t))x^M-T^M(t-\hat{t}_i)x^M_i\|,\\
& &
\|T^N(h(t))x^N\|\}\\
&\overset{\eqref{anis}}{=}& \|T^M(h(t))x^M-T^M(t-\hat{t}_i)x^M_i+T^N(h(t))x^N\|
\end{eqnarray*}
for all $i\geq0$ and $\hat{t}_i\leq t<\hat{t}_{i+1}$ proving
$$
\|T^M(h(t))x^M-x^M_0*t\|\leq \|T(h(t))x-x^M_0*t\|,\quad\quad\forall t\geq0.
$$
Then, \eqref{daryl} implies
$$
\|T^M(h(t))x^M-x^M_0*t\|\leq \epsilon\quad\quad(\forall t\geq0)
\quad\mbox{ and }\quad
\lim_{t\to\infty}\|T^M(h(t))x^M-x^M_0*t\|=0
$$
completing the proof.
\end{proof}

\section{Proof of the theorem}
\label{sec3}

\noindent
First we assume that the Cauchy problem \eqref{eq1} is hyperbolic.
That is, $A$ generates a $C_0$-semigroup $T$ on $X$ with $\sigma(T(1))\cap S^1=\emptyset$.
Then, by Proposition 1.15 p. 305 in \cite{en},
there are a direct sum $X=M\oplus N$ by closed subspaces $M,N$ and numbers
$\lambda_M,\lambda_N,K_M,K_N>0$ such that
$T(t)M= M$, $T(t)N=N$, $T(t)|_N:N\to N$ is invertible (for all $t\geq0$),
\begin{equation}
\label{left}
\|T(t)|_M\|\leq K_Me^{-t\lambda_M}
\quad\mbox{ and }\quad
\|(T(t)|_N)^{-1}\|\leq K_Ne^{-t\lambda_N}
\quad\quad\forall t\geq0.
\end{equation}
Fix $\epsilon>0$.
By \eqref{left} we can apply lemmas \ref{abe} and \ref{lanza} to get
$\delta_M,R_M>0$ and
$\delta_N>0$ satisfying the conclusion of these lemmas for the restrictions $T^M=T|_M$ and $T^N=T|_N$ respectively.
Take
$$
\delta=\min\{\delta_M,\delta_N\}\quad\mbox{ and }\quad
R=\max\{R_M,1\}.
$$

Let $(x_i,t_i)_{i\geq0}$ be a $(\delta,R)$-pseudo orbit 
of $T$ satisfying
\begin{equation}
\label{joe}
\lim_{i\to\infty}\|T(t_i)x_i-x_{i+1}\|=0,
\end{equation}
For all $i\geq0$ one has
$x_i=x^M_i+x^N_i\in M\oplus N$.
By \eqref{anis} one has
$$
\max\{\|T^M(t_i)x^M_i-x^M_{i+1}\|,\|T^N(t_i)x^N_i-x^N_{i+1}\|\}=\|T(t_i)x_i-x_{i+1}\|\leq\delta,\quad\quad\forall i\geq0.
$$
So,
$(x^M_i,t_i)_{i\geq0}$ is a $(\delta_M,R_M)$-pseudo orbit of $T^M$ and $(x^N_i,t_i)_{i\geq0}$ is a $(\delta_N,1)$-pseudo orbit of $T^N$.
Thus, \eqref{joe} together with lemmas \ref{abe} and \ref{lanza} provide $x^M\in M$ and $x^N\in N$ such that
$$
\|T^*(t)x^*-T^*(t-\hat{t}_i)x^*_i\|\leq \epsilon\quad\quad \forall *\in\{M,N\},\,  i\geq0,\, \hat{t}_i\leq t<\hat{t}_{i+1},
$$
and
$$
\lim_{i\to\infty}\|T^*(t_i)x^*_i-x^*_{i+1}\|=0,\quad\quad\forall *\in \{M,N\}.
$$
Therefore, $x=x^M+x^N\in X$ and $h:[0,\infty)\to[0,\infty)$ defined by $h(t)=t$ for $t\geq0$ satisfy
\begin{eqnarray*}
\|T(h(t))x-T(t-\hat{t}_i)x_i\|
&=&\max\{\|T^M(t)x^M-T^M(t-\hat{t}_i)x^M_i\|,\\
& &
\|T^N(t)x^N-T^N(t-\hat{t}_i)x^N\|\}\\
&\leq& \epsilon,
\quad\forall i\geq0,\, \hat{t}_i\leq t<\hat{t}_{i+1},
\end{eqnarray*}
proving $\|T(t)x-x_0*t\|\leq\epsilon$ for all $t\geq0$ and
$$
\lim_{t\to\infty}\|T^*(t)x^*-x_0^* *t\|=0,\quad\quad\forall *\in \{M,N\}
$$
so
$$
\lim_{t\to\infty}\|T(h(t))x-x_0*t\|=\lim_{t\to\infty}\max\{\|T^M(t)x^M-x^M_0*t\|,\|T^N(t)x^N-x_0^N*t\|\}=0
$$
proving that \eqref{eq1} has the L-shadowing property.

Now, we suppose $dim(X)<\infty$ and that \eqref{eq1} has the L-shadowing property.
In such a case, the semigroup $T$ is {\em uniformly continuous} i.e.
$$
\lim_{t\to 0^+}\|T(t)-I\|=0
$$
and so $T(t)=e^{tA}$ for all $t\geq0$ (see \cite{en}).

Suppose by contradiction that \eqref{eq1} is not hyperbolic.
Then,
$\sigma(T(1))\cap S^1\neq\emptyset$ namely there are $x_0\in X$ with $\|x_0\|=1$ and $\lambda\in S^1$ such that
$Ax_0=\lambda x_0$.
Let $M$ be the subspace of $X$ generated by $x_0$.
It follows easily that $Ax=\lambda x$ for all $x\in M$.
Since $T(t)=e^{tA}$ one has that
\begin{equation}
\label{futil}
T(t)x=e^{t\theta i}x,\quad\quad\forall t\geq0, x\in M
\end{equation}
where $\theta=Im(\lambda)$ and $i=\sqrt{-1}$.

Note that this semiflow (in fact, a flow) does not have the L-shadowing property.
This follows from the fact that it is chain transitive (i.e., we can go from one point to another through pseudo orbits) but not transitive (since the orbits are confined to circles).
However, since $dim(X)<\infty$, we can apply the Jordan form to prove that $M$ has a complementary subspace $N$ (i.e. $X=M\oplus N$)
such that $A(N)\subset N$.
This implies
$T(t)M\subset M$ and $T(t)N\subset N$ for all $t\geq0$.
Then, $T|_M$ has the L-shadowing property by Lemma \ref{lanza} which is absurd since \eqref{futil} implies that
$T|_M$ is the semiflow $e^{t\theta i}$.
This completes the proof of the theorem.

\section{Applications}
\label{sec4}

\noindent
Let us give two short applications of the theorem. The first is a spectral sufficient condition
for the Cauchy problem on Hilbert spaces to have the shadowing properties under consideration.

\medskip

\begin{cor}
\label{corina}
If the Cauchy problem \eqref{eq1} in an Hilbert space satisfies
$$
\sigma(A)\cap (i\mathbb{R})=\emptyset\quad\mbox{ and }\quad
\sup_{\lambda\in i\mathbb{R}}\|(\lambda I-A)^{-1}\|<\infty,
$$
then the problem has the L-shadowing property.
\end{cor}

\begin{proof}
The proof follows from theorem and Theorem 1.18 p. 307 in \cite{en}.
\end{proof}

We can also obtain some dynamical conclusion from the shadowing properties.
More precisely, we say that $x\in X$ is a {\em nonwandering point} of the Cauchy problem \eqref{eq1} if
for all neighborhood $U$ of $x$ and $R>0$ there are $y\in U$ and $t\geq R$ such that
$T(t)y\in U$, where $T$ is the semigroup generated by $A$.\\

\begin{prop}
\label{calor}
If the Cauchy problem given by equation \eqref{eq1} is hyperbolic, then $0$ is the only nonwandering point of the problem.
\end{prop}

\begin{proof}
It is clear that $0$ is nonwandering so we just need to prove that any nonwandering point $x$ (say) vanishes.
For this, let us take $\epsilon>0$.
On the other hand, it follows from the proof of the theorem that the problem has the shadowing property (see Remark \ref{walker}).
So, let $\delta,R>0$ with $0<\delta<\epsilon$ be given by this property for $\frac{\epsilon}2$ and the semigroup $T$.
Since $x$ is nonwandering, there are $x_0\in X$ with $\|x_0-x\|<\frac{\delta}2$ and $t_*\geq R$ such that $\|T(t_*)x_0-x\|<\frac{\delta}2$.
Since
$$
\|T(t_*)x_0-x_0\|\leq \|T(t_*)x_0-x\|+\|x-x_0\|<\frac{\delta}2+\frac{\delta}2=\delta,
$$
the sequence $(x_i,t_i)_{i\geq0}$ defined by $(x_i,t_i)=(x_0,t_*)$ for all $i\geq0$ is a $(\delta,R)$-pseudo orbit.
Then, there are $y\in X$ and an increasing homeomorphism $h:[0,\infty)\to [0,\infty)$ such that
\begin{equation}
\label{nigan}
\|T(h(t))y-x_0*t\|\leq\frac{\epsilon}2,\quad\quad\forall t\geq0.
\end{equation}
But $\hat{t}_i=it_*$ so $x_0*t=T(t-it_*)x_0\in \{T(s)x_0:0\leq s\leq t_*\}$ which is bounded thus there is $0<K<\infty$ such that
$\|T(t)y\|<K$ for all $t\geq0$.

By Proposition 1.15 p. 305 in \cite{en},
there are a direct sum $X=M\oplus N$ by closed subspaces $M,N$ and numbers
$\lambda_M,\lambda_N,K_M,K_N>0$ such that
$T(t)M= M$, $T(t)N=N$, $T(t)|_N:N\to N$ is invertible,
$$
\|T(t)|_M\|\leq K_Me^{-t\lambda_M}
\quad\mbox{ and }\quad
\|(T(t)|_N)^{-1}\|\leq K_Ne^{-t\lambda_N}
\quad\quad\forall t\geq0.
$$
Write $y=y^M+y^N\in M\oplus N$.
It follows that
$\{T(t)y^N:t\geq0\}$ is bounded. Since
$\|T(t)y^N\|\geq K^{-1}_Ne^{t\lambda_N}\|y^N\|$ for all $t\geq0$ we conclude that
$y^N=0$. It follows that $y=y^M$ so
$$
\lim_{t\to\infty}T(t)y=0.
$$
But since $x_0*\hat{t}_i=x_0$ for all $i\geq0$, \eqref{nigan} implies $\|x_0\|\leq\frac{\epsilon}2$.
Then, $\|x\|\leq \|x-x_0\|+\|x_0\|\leq \frac{\delta}2+\frac{\epsilon}2<\epsilon$.
Letting $\epsilon\to0$ we get $x=0$ completing the proof.
\end{proof}

Now, consider again the Cauchy problem \eqref{eq1}. Given $\delta,R>0$ we call a sequence
$(x_i,t_0)_{i=0}^k$ {\em $(\delta,R)$-chain} for the problem if $t_i\geq R$ and $\|T(t_i)x_i-x_{i+1}\|<\delta$ ($\forall i\geq0$) where $T$ is the semigroup generated by $A$.
In such a case, we say that the chain is from $x_0$ to $x_k$.
Given $x,y\in X$, we say that $x$ is $(\delta,R)$-related to $y$ if
there are $(\delta,R)$-chains from $x$ to $y$ and from $y$ back to $x$.
We write $x\sim y$ if $x$ is $(\delta,R)$-related to $y$ for all $\delta,R>0$.
The {\em chain recurrent set} is the set of $x\in X$ such that $x\sim x$.

Based on the previous proposition we obtain the following result.

\medskip

\begin{prop}
\label{poco}
If the Cauchy problem given by equation \eqref{eq1} is hyperbolic, then $0$ is the only chain recurrent point of the problem.
\end{prop}

\begin{proof}
The problem is hyperbolic so $0$ is the only nonwandering point by Proposition \ref{calor}.
Then, the result follows since the problem has the L-shadowing property (by the theorem) and this property easily implies that every chain recurrent point is nonwandering. 
\end{proof}

In particular, the hyperbolic Cauchy problems \eqref{eq1} are excluded from the hypercyclic
theory of semigroups \cite{mp}.

We can also apply our theorem to some PDEs.
An example is the heat equation
in a bounded domain $\Omega\subset \mathbb{R}^N$ with boundary $\partial \Omega$,

\begin{equation}
\label{bolso}
\left\{ \begin{array}{rll}
u_t& = & \Delta u,  \hbox{ in }  \Omega\times (0,\infty) \\
u & = & 0,  \,\,\, \,\hbox{ on }  \partial\Omega\times (0,\infty)\\
u(\cdot,0)& = & u_0, \,\, \hbox{ in }  \Omega.
\end{array}
\right.
\end{equation}

\medskip

Transform this equation to the Cauchy problem \eqref{eq1}
with $X=L^2(\Omega)$, and $A:D(A)\to X$ defined by
$$
D(A)=H^2(\Omega)\cap H^1_0(\Omega)\quad\mbox{ and }
\quad
Au=\Delta u,\quad\forall u\in D(A).
$$
See p. 327 in \cite{brezis}.
It follows that $A$ is self-adjoint (hence normal), and $L^2(\Omega)$ is Hilbert so $\sigma(e^{tA})$ is the closure of $\{e^{t\lambda}:\lambda\in\sigma(A)\}$ for all $t>0$ \cite{pru} (here $e^{tA}$ denotes the $C_0$-semigroup generated by $A$). However, $\sigma(A)$ consists of a sequence $0<\lambda_1(\Omega)\leq\lambda_2(\Omega)\leq \lambda_k(\Omega)\leq\cdots\to\infty$
(c.f. Theorem 1.2.2 in \cite{he}), so $\sigma(e^{tA})\cap S^1=\emptyset$ for large $t>0$, thus $e^{tA}$ is hyperbolic for some (and then for all) $t>0$. Therefore, the Cauchy problem is hyperbolic, and it has both the shadowing and L-shadowing properties according to the theorem. As a result, we conclude that the heat equation \eqref{bolso} also possesses the L-shadowing property. It is possible that this result follows from different arguments (e.g. \cite{hen}).

Another equation is
\begin{equation}
\label{naro}
\left\{ \begin{array}{rll}
u_t& = & u_x-\theta u, \hbox{ in }  \mathbb{R}\times (0,\infty) \\
u(\cdot,0)& = & u_0,  \quad \quad \,\,\hbox{ in }  \mathbb{R}
\end{array}
\right.
\end{equation}
with $\theta\neq0$.
Transform this equation into the Cauchy problem
\eqref{eq1} with $X=C_b(\mathbb{R})$ (the set of bounded uniformly continuous maps $u:\mathbb{R}\to \mathbb{R}$ endowed with the supremum norm)
and $A:D(A)\to X$ defined by
$$
D(A)=C^1_b(\mathbb{R})
\quad\mbox{ and }
\quad A(u)=u'-\theta u,\quad\forall u\in D(A),
$$
where $C^1_b(\mathbb{R})$ is the subspace of $C_b(\mathbb{R})$ consisting of all $u$ such that the classical derivative
$$
u'(x)=\lim_{y\to x}\frac{u(y)-u(x)}{y-x}
$$
exists for all $x\in\mathbb{R}$ and belongs to $C_b(\mathbb{R})$.
Since $A$ generates the $C_0$-semigroup
$T$ of $X$ defined by
$$
(T(t)u)x=e^{-t\theta}u(x+t)\quad\quad(\forall x,t\geq0)
$$
which is clearly hyperbolic, one has that the equation has the L-shadowing property by the theorem.

An intriguing question arises from our theorem: Specifically, we wonder whether every Cauchy problem with the shadowing property is hyperbolic. However, our findings suggest that the answer to this question is negative. In fact, this is proven to be false in the case of invertible linear operators on Banach spaces. These operators can possess the shadowing property without being hyperbolic. Recently, a conjecture has emerged, stating that an invertible linear operator exhibits the shadowing property if and only if it is generalized hyperbolic \cite{ddm}. Consequently, it becomes imperative to restate this conjecture within the framework of semigroups of linear operators.

First of all, we say that a linear operator of a Banach space $B:X\to X$ is {\em generalized hyperbolic} if
there are a direct sum $X=M\oplus N$ by closed subspaces $M,N$ and $K,\lambda>0$ such that
$B(M)\subset M$, $N\subset B(N)$,
$$
\|B^nx\|\leq Ke^{-n\lambda}\quad(\forall x\in M,\, n\geq0)
\quad \mbox{ and }
\quad
\|B^nx\|\geq K^{-1}e^{n\lambda}\quad(\forall x\in N,\, n\geq0).
$$
This definition clearly reduces to the original one \cite{cgp} when $B$ is invertible.
As in \cite{bcdmp} it can be proved that every generalized hyperbolic operator $B$ as above has the (positive) shadowing property.

Based on this definition and the description of hyperbolic semigroups in p. 305 of \cite{en} we propose the following definition.

\medskip

\begin{definition}
We say that a semigroup of linear operators $T$ on a Banach space $X$ is {\em generalized hyperbolic} if there are a direct sum $X=M\oplus N$ by closed subspace $M,N$
and $K,\lambda>0$ such that
$T(t)M\subset M$ and $N\subset T(t)N$ ($\forall t\geq0$),
$$
\|T(t)x\|\leq Ke^{-t\lambda}\|x\|\quad(\forall x\in M, t\geq0)
\mbox{ and }
\|T(t)x\|\geq K^{-1}e^{t\lambda}\|x\|\quad(\forall x\in N, t\geq0).
$$
We say that the Cauchy problem \eqref{eq1} is generalized hyperbolic if
$A$ generates a generalized hyperbolic $C_0$-semigroup.
\end{definition}

\medskip

If the Cauchy problem \eqref{eq1} is hyperbolic, then it is generalized hyperbolic.
The converse is false by the following proposition.

\medskip

\begin{prop}
There is a Cauchy problem \eqref{eq1} which is generalized hyperbolic but not hyperbolic.
\end{prop}

\begin{proof}
Define $w:\mathbb{R}\to\mathbb{R}$ by $w(x)=e^{|x|}$ for all $x\in\mathbb{R}$.
Consider the Cauchy problem \eqref{eq1} in $X=L^2(\mathbb{R})$ with
$A:D(A)\to L^2(\mathbb{R})$ defined by
$$
D(A)=H^1(\mathbb{R}):=\{u\in L^2(\mathbb{R}):u'\in L^2(\mathbb{R})\}
\quad\mbox{ and }\quad A(u)=u'-\frac{w'}wu,\quad\forall u\in D(A)
$$
(derivatives in the weak sense).

Note that $A$ generates the $C_0$-semigroup
$$
(T(t)u)(x)=\frac{w(x)}{w(x+t)}u(x+t),\quad\quad\forall u\in L^2(\mathbb{R}),\, x\in \mathbb{R},\, t\geq0.
$$
To see this fix $u\in H^1(\mathbb{R})$. Given $\epsilon>0$ one has $1\leq e^t<1+\epsilon$ for $t\geq0$ small so straightforward computations show
\begin{eqnarray*}
\left\|
\frac{T(t)u-u}t-u'+\frac{w'}wu\right\|^2_{L^2} & \leq &
e^t\int\left|\frac{u(x+t)-u(x)}t-u'(x)\right|^2dx
+\epsilon^2\|u'\|^2_{L^2}+\\
& & \|u\|^2_{L^2}\cdot \int\left|
\frac{w(x+t)-w(x)}t-w'(x)\right|^2dx+\\
& & (e^t-1)\|u\|_{L^2}^2.
\end{eqnarray*}
Then, the assertion follows by letting $\epsilon,t\to0$ since the $L^2$ and weak derivatives coincide (e.g. Exercise 1.9 p. 21 in \cite{lp}).

Now, define
$$
M=\{u\in L^2(\mathbb{R}):u=0\mbox{ a.e. in }[0,\infty)\}
\mbox{ and }
N=\{u\in L^2(\mathbb{R}):u=0\mbox{ a.e. in }(-\infty,0]\}.
$$
Then, $M$ and $N$ are closed subspaces with $L^2(\mathbb{R})=M\oplus N$.
Clearly
$T(t)M\subset M$ for all $t\geq0$.
On the other hand, each $T(t)$ is invertible with inverse
$T(t)^{-1}:L^2(\mathbb{R})\to L^2(\mathbb{R})$ given by
$$
(T(t)^{-1}u)(x)=\frac{w(x)}{w(x-t)}u(x-t),\quad\quad\forall x\in\mathbb{R}.
$$
We can easily check $T(t)^{-1}N\subset N$ hence $N\subset T(t)N$ for all $t\geq0$.
Since for all $u\in M$ and $t\geq0$
$$
\|T(t)u\|^2_{L^2}=\int_{-\infty}^\infty \left(\frac{w(x)}{w(x+t)}\right)^2|u|^2dx
=\int_0^\infty e^{-2t}|u|^2dx
=e^{-2t}\|u\|^2_{L^2},
$$
one has
$\|T(t)u\|_{L^2}=e^{-t}\|u\|_{L^2}$ for all $u\in M$ and $t\geq0$. Likewise,
$\|T(t)^{-1}u\|_{L^2}=e^{-t}\|u\|_{L^2}$ for all $u\in N$  and $\forall t\geq0$. We then conclude that $T$ (and therefore its corresponding Cauchy problem) are generalized hyperbolic.

To finish, let us prove that the problem is not hyperbolic. Indeed,
take any continuous function $u:\mathbb{R}\to\mathbb{R}$ with $u\neq0$
whose support
$supp(u)=\{x\in \mathbb{R}:u(x)\neq0\}$ is contained in a compact interval of $(0,\infty)$.
Clearly $u\in N$ so
$T(t)^{-1}u\to0$ in $L^2(\mathbb{R})$ as $t\to\infty$.
It follows also that $\{x\in \mathbb{R}:u(x+t)\neq0\}=supp(u)-t$ for all $t\geq0$ so
$T(t)u\in M$ for $t>0$ large.
Then, $T(t)u\to 0$ in $L^2(\mathbb{R})$ as $t\to\infty$. This implies that $u$ is recurrent, and
so, the problem cannot be hyperbolic by Proposition \ref{poco}. This completes the proof.
\end{proof}

\medskip

Just as every generalized hyperbolic invertible operator in a Banach space possesses the shadowing and L-shadowing properties \cite{bcdmp,lm}, we believe that the same holds true for every generalized hyperbolic Cauchy problem. If this belief is valid, it implies that the proposition mentioned earlier offers an example of a nonhyperbolic Cauchy problem that exhibits both shadowing properties.

\section*{Acknowledgements}
We thank the Sejong Institute for Mathematical Sciences (SIMS), Sejong, South Korea, for its warm hospitality during the preparation of this work.

\section*{Funding}
Partially supported by 
the Basic Science Research Program through the NRF funded by the Ministry of Education of the Republic of Korea (Grant Number: 2022R1l1A3053628).

\section*{Declaration of competing interest}

\noindent
There is no competing interest.

\section*{Data availability}

\noindent
No data was used for the research described in the article.


\begin{thebibliography}{0000}




\bibitem{ah}
Aoki, N., Hiraide, K.,
{\em Topological theory of dynamical systems.
Recent advances},
North-Holland Mathematical Library, 52. North-Holland Publishing Co., Amsterdam, 1994.



\bibitem{bvr}
Baskakov, A.G., Vorobev, A.A., Romanova, M.Y.,
Hyperbolic operator semigroups and the Lyapunov equation,
{\em Math. Notes} 89 (2011), no. 1-2, 194--205. 






\bibitem{bt}
Batty, C.J.K., Tomilov, Y.,
Quasi-hyperbolic semigroups,
{\em J. Funct. Anal.} 258 (2010), no. 11, 3855--3878.


\bibitem{bcdmp}
Bernardes, N.C., Cirilo, P.R. , Darji, U.B., Messaoudi, A.,Pujals, E.R.,
Expansivity and shadowing in linear dynamics,
{\em J. Math. Anal. Appl.} 461 (2018), 796--816. 



\bibitem{brezis}
Brezis, H.,
{\em Functional analysis, Sobolev spaces and partial differential equations},
Universitext. Springer, New York, 2011. 


\bibitem{b}
Bowen, R.,
$\omega$-limit sets for axiom A diffeomorphisms,
{\em J. Differential Equations} 18 (1975), no. 2, 333--339.



\bibitem{cc}
Carvalho, B., Cordeiro, W.,
Positively $N$-expansive homeomorphisms and the L-shadowing property,
{\em J. Dynam. Differential Equations} 31 (2019), no. 2,
1005--1016.


\bibitem{cgp}
Cirilo, P.,  Gollobit, B.,  Pujals, E.,
Dynamics of generalized hyperbolic linear operators,
{\em Adv. Math.} 387 (2021), Paper No. 107830, 37 pp. 


\bibitem{ddm}
D'Aniello, E., Darji, U.B., Maiuriello, M.,
Generalized Hyperbolicity and Shadowing in $L^p$ spaces,
{\em J. Differential Equations} 298 (2021), 68--94. 




\bibitem{dln}
Du, N.H., Lee, K., Nguyen, N.,
Structural stability and L-shadowing for flows,
Preprint (2023).



\bibitem{en}
Engel, K-J., Nagel, R.,
{\em One-parameter semigroups for linear evolution equations},
Graduate Texts in Mathematics, 194. Springer-Verlag, New York, 1991.


\bibitem{fs}
Franke, J.E., Selgrade, J.F.,
Hyperbolicity and chain recurrence,
{\em J. Differential Equations} 26 (1977), no. 1, 27--36.

\bibitem{g}
Gao, S.,
A shadowing lemma for random dynamical systems,
{\em Journal of Applied Analysis and Comuputation} 11 (2021), 3014-3030.



\bibitem{gd}
Garg, M., Das, R.,
Continuous semi-flows with the almost average shadowing property,
{\em Chaos Solitons Fractals} 105 (2017), 1--7. 




\bibitem{hwl}
He, L.F., Wang, Z.H., Li, H.,
Continuous semiflows with the shadowing property,
{\em Acta Math. Appl. Sinica} 19 (1996), no. 2, 297--303.


\bibitem{he}
Henrot, A.,
{\em Extremum problems for eigenvalues of elliptic operators}, Frontiers in Mathematics. Birkh\"auser Verlag, Basel, 2006.



\bibitem{hen}
Henry, D.,
{\em Geometric theory of semilinear parabolic equations},
Lecture Notes in Mathematics, 840. Springer-Verlag, Berlin-New York, 1981. 



\bibitem{kh}
Katok, A., Hasselblatt, B.,
{\em Introduction to the modern theory of dynamical systems. With a supplementary chapter by Katok and Leonardo Mendoza},
Encyclopedia of Mathematics and its Applications, 54. Cambridge University Press, Cambridge, 1995.


\bibitem{k}
Komuro, M.,
One-parameter flows with the pseudo-orbit tracing property, {\em Monatsh. Math.} 98 (1984), 219--253.



\bibitem{lm}
Lee, K., Morales, C.A.,
Topological stability, L-shadowing, and generalized hyperbolicity,
Preprint (2023).





\bibitem{ls}
Lee, K., Sakai, K.,
Various shadowing properties and their equivalence,
{\em Discrete Contin. Dyn. Syst.} 13 (2005), no. 2, 533--540.



\bibitem{lp}
Linares, F., Ponce, G.,
{\em Introduction to nonlinear dispersive equations},
Universitext. Springer, New York, 2009.


\bibitem{mp}
Mangino, E.M., Peris, A.,
Frequently hypercyclic semigroups,
{\em Studia Math.} 202 (2011), no. 3, 227--242. 


\bibitem{o}
Ombach, J.,
The simplest shadowing,
{\em Ann. Polon. Math.} 58 (1993), 253--258.

\bibitem{o1}
Ombach, J.,
The shadowing lemma in the linear case,
{\em Univ. Iagell. Acta Math.} 21 (1994), 69--74.


\bibitem{pazy}
Pazy, A.,
{\em Semigroups of linear operators and applications to partial differential equations},
Applied Mathematical Sciences, 44. Springer-Verlag, New York, 1983.


\bibitem{pil}
Pilyugin, S.Y.,
Shadowing in the Chafee-Infante problem,
{\em St. Petersburg Math. J.} 12 (2001), no. 4, 681--711.

\bibitem{pru}
Pr\"uss, J.,
On the spectrum of $C_0$-semigroups,
{\em Trans. Amer. Math. Soc.} 284 (1984), no. 2, 847--857. 



\bibitem{ra}
Rau, R.,
Hyperbolic evolution semigroups on vector valued function spaces,
{\em Semigroup Forum} 48 (1994), no. 1, 107--118. 

\bibitem{r}
Rojas, A.,
Stability of geometric separating flows,
Preprint (2023).




\bibitem{t}
Thomas, R.F.,
Stability properties of one-parameter flows,
{\em Proc. London Math. Soc.} 45 (1982), 479--505.



\end{thebibliography}
\end{document}